\documentclass[preprint,11pt]{elsarticle}

\usepackage{amssymb,amsmath,amsthm}
\usepackage{graphicx}
\usepackage{float}
\restylefloat{figure}

\newcommand{\beq}[1]{ \begin{equation}\label{#1} }
\newcommand{\eeq}{\end{equation}}
\numberwithin{equation}{section}
\DeclareMathOperator*{\esssup}{{\rm ess}\,sup}

\usepackage{fullpage}
\newtheorem{theorem}{Theorem}
\newtheorem{lemma}{Lemma}
\newtheorem{corollary}{Corollary}
\newtheorem{definition}{Definition}
\newtheorem{example}{Example}

\newtheorem{proposition}{Proposition}
\journal{Applied Mathematics Letters}

\begin{document}

\begin{frontmatter}

\title{Solution estimates and stability tests for linear neutral differential equations}

\author[label1]{Leonid Berezansky}
\author[label2]{Elena Braverman}
\address[label1]{Dept. of Math.,
Ben-Gurion University of the Negev,
Beer-Sheva 84105, Israel}
\address[label2]{Dept. of Math. and Stats., University of
Calgary,2500 University Drive N.W., Calgary, AB, Canada T2N 1N4; e-mail
maelena@ucalgary.ca, phone 1-(403)-220-3956, fax 1-(403)--282-5150 (corresponding author)}

\begin{abstract}
Explicit exponential stability tests are obtained for the scalar neutral differential equation
$$
\dot{x}(t)-a(t)\dot{x}(g(t))=-\sum_{k=1}^m b_k(t)x(h_k(t)),
$$
together with exponential estimates for its solutions.

Estimates for solutions of a non-homogeneous neutral equation are also obtained, they are
valid on every finite segment, 
thus describing both asymptotic and transient behavior. For neutral differential equations, 
exponential estimates are obtained here for the first time. 
Both the coefficients and the delays are assumed to be measurable, not necessarily continuous functions. 
\end{abstract}


\begin{keyword}
linear neutral differential equations, exponential stability tests,
explicit solution estimates, variable delays and coefficients

\noindent
{\bf AMS subject classification:} 
34K20, 34K25, 34K06
\end{keyword}

\end{frontmatter}

\section{Introduction}

In  recent papers \cite{BB1,BB2} we obtained new exponential stability conditions
for a scalar linear neutral differential equation 
\begin{equation}\label{4.5}
\dot{x}(t)-a(t)\dot{x}(g(t))=-b(t)x(h(t)), ~~t \geq t_0.
\end{equation}
Here the functions $a, b, g, h$ are assumed to be Lebesgue measurable, $b$ is essentially bounded
on $[t_0,\infty)$, $a$ satisfies
\begin{equation}
\label{a_less_1}
|a(t)|\leq A_0<1, ~~ t\geq t_0.
\end{equation}
The condition on the delay in the neutral term
\begin{equation}
\label{del_neut}
mes~ E=0\Longrightarrow mes~ g^{-1}(E)=0,
\end{equation}
where $mes~E$ is  the Lebesgue
measure of the set $E$, guarantees that $u(g(t))$ is properly defined and is  Lebesgue measurable for any measurable $u$.
The delays in both terms of (\ref{4.5}) are variable but bounded: for some $\tau>0$ and $\sigma>0$, 
$0\leq t-g(t)\leq \sigma$, $0 \leq t-h(t) \leq \tau$ for $t \geq t_0$.
All the functions are considered in the space $L_{\infty}$ of Lebesgue measurable essentially bounded functions with the essential supremum norm
$\| \cdot \|_J$ on a certain segment $J \subset [t_0- \max\{\tau,\sigma\}, \infty)$.

\begin{proposition}\label{proposition5.1} \cite{BB1}
If at least one of the following conditions holds
\begin{equation}\label{5.1}
0<b_0\leq b(t),~\|b\|_{[t_0,\infty)}\tau \leq \frac{1}{e}\,,~
\|a\|_{[t_0,\infty)}+\|b\|_{[t_0,\infty)}\left\|\frac{a}{b}\right\|_{[t_0,\infty)}<1;
\end{equation}
 \begin{equation}\label{5.2}
\left\|\frac{a}{b_1}\right\|_{[t_0,\infty)}\frac{\|b\|_{[t_0,\infty)}}{1-\|a\|_{[t_0,\infty)}}
+\left\|\frac{b-b_1}{b_1}\right\|_{[t_0,\infty)} < 1, \mbox{ where ~} b_1(t):= \min \left\{ b(t),\frac{1}{\tau e}\right\};
\end{equation}
\begin{equation}\label{5.3}
 \displaystyle \| b\|_{[t_0,\infty)} \left(  \left\|\frac{a}{b}\right\|_{[t_0,\infty)}
 + \left\| \left( t-h(t)- \frac{1}{\|b\|_{[t_0,\infty)}e} \right)^+ \right\|_{[t_0,\infty)} \right) < 1- \| a \|_{[t_0,\infty)},
\end{equation}
where $u^+=\max\{u,0\}$, equation (\ref{4.5}) is uniformly exponentially stable.
\end{proposition}

\begin{proposition}\label{proposition5.2} \cite[Theorem~1(a)]{BB2}
Assume that $0\leq a_0\leq a(t)\leq A_0<1$, $0< b_0\leq b(t)\leq B_0$ and $\displaystyle
\tau B_0+\frac{\sigma A_0 B_0^2 (1-a_0)}{(1-A_0)^2 b_0}<1-A_0$.
Then equation (\ref{4.5}) is uniformly exponentially stable.
\end{proposition}

In the present paper, we consider a generalization of (\ref{4.5}) to the case of several delays in the non-neutral part
\begin{equation}\label{1.1}
\dot{x}(t)-a(t)\dot{x}(g(t))=-\sum_{k=1}^m b_k(t)x(h_k(t)).
\end{equation}
We refer the reader to the review of known stability results for neutral equations in \cite{BB1}. 
For a particular case of constant coefficients and delays sharp results are obtained in \cite{Balandin}.
The recent paper \cite{BB3} contains exponential estimates for solutions of delay differential equations without
the neutral term ($a(t)\equiv 0$ in \eqref{1.1}), see also \cite{Khar}.
The purpose of the present paper is two-fold.
\begin{enumerate}
\item
Investigate stability of (\ref{1.1}) and get sufficient conditions which involve all the delays in the non-neutral part.
Among other results, we obtain a new exponential stability test which has a very simple form and can be applied to a wide
class of neutral equations.
\item
Develop explicit exponential estimates for solutions of (\ref{1.1}) and its non-homogeneous version, dependent on the right-hand side and the initial functions. These inequalities are valid on every finite segment, thus both describing asymptotic and transient behavior.
For neutral differential equations, exponential  estimates are obtained here for the first time.
\end{enumerate}

Note that our assumptions refer to a half-line $[t_0,\infty)$, which is essential for exponential stability tests. However, once all the parameters are considered on a finite segment $[t_0,t_1]$, solution estimates on this interval remain valid. 
Only few asymptotic formulas for solutions of neutral differential equations are known, see \cite{Dom,Kordonis,Matveeva}
and references therein, and the present paper fills this gap.

The paper is organized as follows. Section 2 includes some definitions and auxiliary results, among them 
an estimate for the fundamental function of an  equation with
a non-delay term. Section 3 is the main part of the paper
where we obtain stability tests and develop solution estimates for equation \eqref{1.1} and its non-homogeneous version.
Section 4 presents illustrating examples and discussion. 

\section{Preliminaries}

We consider scalar delay differential equation (\ref{1.1})
where, similarly to (\ref{4.5}), $a, b_k, g, h_k$ are Lebesgue measurable, $a$ satisfies (\ref{a_less_1}) and
$b_0\leq b_k(t)\leq B_k$  for $t\geq t_0\geq 0$, implication (\ref{del_neut}) holds and there are $\sigma>0$, $\tau_k>0$ 
such that $0\leq t-g(t)\leq \sigma$, $0 \leq t-h_k(t) \leq \tau_k$ for $t \geq t_0$, $i=1, \dots, m$.
 

Along with  (\ref{1.1}), we consider an initial value problem for a non-homogeneous equation
\begin{equation}
\label{2.1}
\dot{x}(t)-a(t)\dot{x}(g(t))+\sum_{k=1}^m b_k(t)x(h_k(t))=f(t), ~t\geq t_0,~
\end{equation}
\begin{equation}
\label{2.2}
x(t)=\varphi(t), ~ t \leq t_0,~\dot{x}(t)=\psi(t),~ t<t_0, 
\end{equation}
where $f:[t_0,\infty)\rightarrow {\mathbb R}$ is a Lebesgue measurable locally essentially bounded  function,
$\varphi:(-\infty,t_0] \rightarrow {\mathbb R}$ and $\psi :(-\infty,t_0)\rightarrow {\mathbb R}$ 
are Borel measurable bounded functions.

Further, we assume that the above conditions hold for (\ref{1.1}) and (\ref{2.1})-(\ref{2.2}) without mentioning it,
as well as similar conditions for all other neutral equations considered in the paper.

\begin{definition} 
A locally absolutely continuous on $[t_0,\infty)$
function $x: {\mathbb R} \rightarrow {\mathbb R}$ is called {\bf a solution of problem} (\ref{2.1})-(\ref{2.2}) 
if it satisfies equation (\ref{2.1}) for almost all $t\in [t_0,\infty)$ and
the equalities in (\ref{2.2})
for $t\leq t_0$.
%
For each $s\geq t_0$, the solution $X(t,s)$ of the problem
\begin{equation}
\label{2.3}
\dot{x}(t)-a(t)\dot{x}(g(t))+\sum_{k=1}^m b_k(t)x(h_k(t))=0, ~x(t)=0,~\dot{x}(t)=0,~t<s,~x(s)=1
\end{equation}
is called {\bf the fundamental function} of equation  (\ref{1.1}). We assume $X(t,s)=0$ for $0\leq t<s$.
\end{definition}

\begin{definition}
We will say that equation (\ref{1.1}) is {\bf uniformly exponentially stable} 
if there exist 
$M>0$ and $\gamma>0$ such that 
the solution of problem  (\ref{2.1})-(\ref{2.2})  with $f \equiv 0$
has the estimate 
$\displaystyle 
|x(t)|\leq M e^{-\gamma (t-t_0)} \sup_{t \in (-\infty,  t_0]}(|\varphi(t)|+|\psi(t)|)$, $t\geq t_0$,
where $M$ and $\gamma$ do not depend on $t_0 \geq 0$, $\varphi$ and $\psi$.
The fundamental function $X(t,s)$ of equation (\ref{1.1}) {\bf has an exponential estimate} if it satisfies
$\displaystyle |X(t,s)|\leq M_0 e^{-\gamma_0(t-s)}$ for some $t_0\geq 0$, $M_0>0$, $\gamma_0>0$ and $t\geq s\geq t_0$.
\end{definition}
 
For a fixed bounded interval $J=[t_0,t_1]$, consider the space $L_{\infty}[t_0,t_1]$ of all essentially bounded on $J$
functions with the 
norm $\|y\|_J= \esssup_{t\in J} |y(t)|$,
also $\|f\|_{[t_0,\infty)}=\esssup_{t\geq t_0} |f(t)|$, 
$I$ is the identity operator.
Define a linear bounded operator 
on the space $L_{\infty}[t_0,t_1]$ as 
$\displaystyle 
(Sy)(t)=\left\{\begin{array}{ll}
a(t)y(g(t)),& g(t)\geq t_0,\\
0,& g(t)<t_0.\\
\end{array}\right. 
$
Note that there exists a unique
solution of
problem (\ref{2.1})-(\ref{2.1}), see, for example, \cite{AzbSim},
and it 
can be presented as 
\begin{equation}
\label{star1}
\begin{array}{ll}
x(t)  = & \displaystyle X(t,t_0)x_0+\int_{t_0}^t X(t,s)[(I-S)^{-1}f](s)ds 
+  \int_{t_0}^{t_0+\sigma} X(t,s)[(I-S)^{-1}(a(\cdot)\psi(g(\cdot)))](s)ds
\\ & \displaystyle  -\sum_{k=1}^m \int_{t_0}^{t_0+\tau_k} X(t,s)[(I-S)^{-1}(b_k(\cdot)\varphi(h_k(\cdot)))](s)ds,
\end{array}
\end{equation}
where 
 $\psi(g(t))=0$ for $g(t)\geq t_0$,
$\varphi(h_k(t))=0$ for $h_k(t)\geq t_0 $, and in $L_{\infty}[t_0,t_1]$, for any $t_1>t_0$, 
\begin{equation}
\label{star} 
\|(I-S)^{-1}\|\leq \frac{1}{1-\|a\|_{[t_0,\infty)}}.
\end{equation}


Let us start with a uniform estimate
\begin{equation}\label{3.2}
|Y(t,s)|\leq K, ~~t \geq s \geq t_0
\end{equation}
for  the fundamental function $Y(t,s)$ of the  equation 
with a non-delay term
\begin{equation}\label{3.1}
\dot{y}(t)-a_0(t)\dot{y}(g(t))=c(t)y(t) 
- \sum_{k=0}^m d_k(t)y(h_k(t)), ~t \geq t_0,
\end{equation}
where  $0\leq t-g(t)\leq \sigma$, $t- h_k(t)\leq \tau_k$, $k=0, \dots, m$.
Denote 
$\displaystyle d(t):=\sum_{k=0}^m d_k(t)$.

\begin{lemma}\label{lemma3.1}
If  $\|a_0\|_{[t_0,\infty)}<1$,  there is an $\alpha_0 >0$ such that $d(t)-c(t)\geq \alpha_0$ and
\begin{equation}\label{3.3}
K_0:=\left(\frac{\|c\|_{[t_0,\infty)}+\sum_{k=0}^m \|d_k\|_{[t_0,\infty)}}{1-\|a_0\|_{[t_0,\infty)}}\right)
\left(\left\|\frac{a_0}{d-c}\right\|_{[t_0,\infty)}+\sum_{k=0}^m  \tau_k \left\|\frac{d_k}{d-c}\right\|_{[t_0,\infty)}\right)<1
\end{equation}
 then the fundamental function $Y(t,s)$ of (\ref{3.1}) satisfies (\ref{3.2}) with $K= (1-K_0)^{-1}.$\end{lemma}
\begin{proof}
For brevity of notations, we set $y(t)=Y(t,t_0)$. Then, $y$ satisfies (\ref{3.1}), where the initial value is $y(t_0)=1$, with the zero initial functions.
Let $J=[t_0,t_1]$, where $t_1>t_0$ is arbitrary. Equality (\ref{3.1}) implies the estimate, due to (\ref{star}),
\begin{equation}\label{3.4}
 \|\dot{y}\|_J\leq \left(\frac{\|c\|_{[t_0,\infty)}+\sum_{k=0}^m \|d_k\|_{[t_0,\infty)}}{1-\|a_0\|_{[t_0,\infty)}}\right)\|y\|_J.
\end{equation}
Further, since $d(t)y(t)-\sum_{k=0}^m d_k(t)y(h_k(t))= \sum_{k=0}^m d_k(t)[y(t)-y(h_k(t))]$, from \eqref{3.1},
$$
\dot{y}(t)=-[d(t)-c(t)]y(t)+a_0(t)\dot{y}(g(t))+\sum_{k=0}^m d_k(t)\int_{h_k(t)}^t \dot{y}(\xi)d\xi.
$$
Integrating from $t_0$ to $t$, we get 
\begin{align*}
y(t)= & e^{-\int_{t_0}^t [d(\xi)-c(\xi)]d\xi}+\int_{t_0}^t e^{-\int_s^t [d(\xi)-c(\xi)]d\xi}[d(s)-c(s)]
\times
\\
 & \times\left [\frac{a_0(s)}{d(s)-c(s)}\dot{y}(g(s))+\sum_{k=0}^m \frac{d_k(s)}{d(s)-c(s)}\int_{h_k(s)}^s \dot{y}(\xi)d\xi\right]ds.
\end{align*}
Therefore, by (\ref{3.4}) and the definition of $K_0$ in (\ref{3.3}), we have
$$
\|y\|_J\leq 1+\left(\left\|\frac{a_0}{d-c}\right\|_{[t_0,\infty)}+\sum_{k=0}^m \tau_k\left\|\frac{d_k}{d-c}\right\|_{[t_0,\infty)}\right)\|\dot{y}\|_J
\leq 1+ K_0 \|y\|_J.
$$
Then
$ \|Y(t,t_0)\|_J\leq (1-K_0)^{-1}$, and the expression in 
the right-hand side 
does not depend on $t_1$. 

Hence 
 $ \displaystyle \|Y(t,t_0)\|_{[t_0,\infty)}\leq  (1-K_0)^{-1}$.
Again, the same inequality holds with $t_0$ replaced by any $s \geq t_0$. Thus estimate (\ref{3.2}) holds.
\end{proof}

\section{Main Results}

We start with an exponential estimate which later will be used to analyze exponential stability.

\begin{theorem}\label{theorem4.1}
Assume that there exist constants $\lambda>0$ and $\alpha>0$ such that
\begin{equation}\label{4.1}
p(t):=\sum_{k=1}^m e^{\lambda(t-h_k(t))}b_k(t)+\lambda a(t)e^{\lambda(t-g(t))}-\lambda\geq \alpha, ~ t \geq t_0, ~e^{\lambda\sigma}\|a\|_{[t_0,\infty)}<1,
\end{equation}
\begin{equation}
\label{4.2}
M_1:=
\frac{\lambda+\sum\limits_{k=1}^m e^{\lambda\tau_k}\|b_k\|_{[t_0,\infty)}+\lambda e^{\lambda\sigma}\|a\|_{[t_0,\infty)}}{1-e^{\lambda\sigma}\|a\|_{[t_0,\infty)}} 
\left(\left\|\frac{a}{p}\right\|_{[t_0,\infty)} \!\!\! \!\!\! \!\!\!\! (1+\lambda\sigma)e^{\lambda\sigma}
+\sum_{k=1}^m  \left\|\frac{ b_k}{p}\right\|_{[t_0,\infty)} \!\!\!\!\!\!\!\!\!\! e^{\lambda\tau_k}\tau_k \right)<1.
\end{equation}
Then for the solution of problem ({2.1}),  the following estimate is valid
\begin{equation}
\label{4.3}
\begin{array}{ll} |x(t)|\leq & \displaystyle M_0e^{-\lambda(t-t_0)}\left[|x(t_0)|
+\frac{e^{\lambda \sigma}-1}{\lambda(1-\|a\|_{[t_0,\infty)})} \|a\|_{[t_0,\infty)}\|\psi\|_{[t_0-\sigma,t_0]}\right.
\vspace{2mm}
\\
& \displaystyle \left.+\sum_{k=1}^m\frac{e^{\lambda \tau_k}-1}{\lambda(1-\|a\|_{[t_0,\infty)})} \|b_k\|_{[t_0,\infty)}\|\varphi\|_{[t_0-\tau_k,t_0]}\right]+\frac{M_0}{\lambda(1-\|a\|_{[t_0,\infty)})}\|f\|_{[t_0,t]}, \end{array}
\end{equation}
where $M_0:=(1-M_1)^{-1}$. 
\end{theorem}
\begin{proof}
Consider first the case $f \equiv 0$.
After the substitution $x(t)=e^{-\lambda(t-t_0)}z(t)$ into (\ref{2.1}), we get
\begin{equation}\label{4.4}
\dot{z}(t)-a(t)e^{\lambda(t-g(t))}\dot{z}(g(t))=\lambda z(t)-\lambda e^{\lambda(t-g(t))}a(t)z(g(t))-\sum_{k=1}^m e^{\lambda(t-h_k(t))}b_k(t)z(h_k(t)).
\end{equation}
Equation (\ref{4.4}) has the form of (\ref{3.1}) with 
$$
a_0(t)=a(t)e^{\lambda(t-g(t))}, ~c(t)=\lambda, ~d_0(t)=\lambda a(t)e^{\lambda(t-g(t))},  ~h_0(t)=g(t),
~d_k(t)=e^{\lambda(t-h_k(t))}b_k(t),$$ $k=1,\dots,m$.
Again, $\displaystyle d(t)=\sum_{k=0}^m d_k(t)$, in (\ref{4.1}) we have $p=d-c$.

Then 
$$
\left\|\frac{a_0}{d-c}\right\|_{[t_0,\infty)} \!\!\!\! \leq \left\|\frac{a}{p}\right\|_{[t_0,\infty)}\!\!\!\! e^{\lambda\sigma},
\left\|\frac{d_0}{d-c}\right\|_{[t_0,\infty)}\!\!\!\! \leq \left\|\frac{a}{p}\right\|_{[t_0,\infty)}\!\!\!\!\lambda e^{\lambda\sigma},
\left\|\frac{d_k}{d-c}\right\|_{[t_0,\infty)}\!\!\!\!\leq \left\|\frac{b_k}{p}\right\|_{[t_0,\infty)}\!\!\!\!e^{\lambda\tau_k}, k=1,\dots,m,
$$$$
\frac{\|c\|_{[t_0,\infty)}+\sum_{k=0}^m \|d_k\|_{[t_0,\infty)}}{1-\|a_0\|_{[t_0,\infty)}}\leq
\frac{\lambda+\sum_{k=1}^m e^{\lambda\tau_k}\|b_k\|_{[t_0,\infty)}+\lambda e^{\lambda\sigma}\|a\|_{[t_0,\infty)}}{1-e^{\lambda\sigma}\|a\|_{[t_0,\infty)}}.
$$

Let $Z(t,s)$ be the fundamental function of (\ref{4.4}).
Inequalities (\ref{4.1}) and (\ref{4.2}) imply  (\ref{3.3}).  
By Lemma \ref{lemma3.1}, $|Z(t,s)|\leq M_0$. If $X(t,s)$ is
a fundamental function  of (\ref{2.1}) then for $X(t,s)$  the exponential equality $X(t,s)=e^{-\lambda(t-s)}Z(t,s)$ is valid.
Hence $|X(t,s)| \leq M_0 e^{-\lambda(t-s)}$. 

By (\ref{star1}), the solution $x$ of (\ref{2.1})-(\ref{2.2}) satisfies
\begin{align*}
 |x(t)|\leq & |X(t,t_0)| \, |x_0|
+  \int_{t_0}^{t_0+\sigma} |X(t,s)|\, \|(I-S)^{-1}\| \, |a(s)| \, |\psi(g(s))|ds
\\
& +\sum_{k=1}^m \int_{t_0}^{t_0+\tau_k} |X(t,s)| \, \|(I-S)^{-1}\| \, |b_k(s)| \, |\varphi(h_k(s))| \, ds
\\
\leq  & M_0 e^{-\lambda(t-t_0)} |x(t_0)|+\frac{M_0}{\lambda\left(1- \|a\|_{[t_0,\infty)}\right)} 
 \|a\|_{[t_0,\infty)} \left( e^{-\lambda (t-t_0-\sigma)}-e^{-\lambda(t-t_0)} \right)\|\psi\|_{[t_0-\sigma,t_0]}
\\
& +\frac{M_0}{\lambda\left(1-\|a\|_{[t_0,\infty)}\right)} 
\sum_{k=1}^m \|b_k\|_{[t_0,\infty)} \left( e^{-\lambda (t-t_0-\tau_k)}
-e^{-\lambda(t-t_0)} \right)\|\varphi\|_{[t_0-\tau_k,t_0]},
\end{align*}
which implies  (\ref{4.3}) with $f \equiv 0$.

For the general case we apply (\ref{star1}), the estimate for $X(t,s)$ and the inequalities
$$
\left| \int_{t_0}^t X(t,s) (I-S)^{-1}(f(s))\, ds \right| 
\leq \frac{M_0}{\lambda\left(1- \|a\|_{[t_0,\infty)}\right)} \|f\|_{[t_0,t]}.
$$

\end{proof}
From continuity of $p$ in $\lambda$, where $p$ is defined in (\ref{4.1}),
Theorem \ref{theorem4.1} immediately implies the following exponential stability test.

\begin{theorem}\label{theorem4.2}
Assume that for some $\alpha>0$, $\displaystyle b(t):=\sum_{k=1}^m b_k(t)\geq \alpha,~\|a\|_{[t_0,\infty)}<1$ and
$$
\left(\frac{\sum_{k=1}^m \|b_k\|_{[t_0,\infty)}}{1-\|a\|_{[t_0,\infty)}} \right)
\left(\left\|\frac{a}{b}\right\|_{[t_0,\infty)}+\sum_{k=1}^m \tau_k\left\|\frac{b_k}{b}\right\|_{[t_0,\infty)}\right)<1.
$$
Then equation (\ref{1.1}) is uniformly exponentially stable.
\end{theorem}

\begin{corollary}\label{corollary4.1}
Assume that $\displaystyle  b_k>0$ are  constants, 
$\displaystyle
\|a\|_{[t_0,\infty)}<\frac{1}{2}$ and $\displaystyle\sum_{k=1}^m b_k\tau_k<1-2\|a\|_{[t_0,\infty)}.
$
Then equation (\ref{1.1}) is uniformly exponentially stable.
\end{corollary}
Consider (\ref{4.5}) which is equation (\ref{1.1}) for $m=1$.

\begin{corollary}\label{corollary4.2}
Assume that for some $\alpha>0$, $b(t)\geq \alpha$ for any $t \geq t_0$, $\|a\|_{[t_0,\infty)}<1$, and
$$
\left(\left\|\frac{a}{b}\right\|_{[t_0,\infty)}+\tau\right)\|b\|_{[t_0,\infty)}<1-\|a\|_{[t_0,\infty)}.
$$
Then equation (\ref{4.5}) is uniformly exponentially stable.
\end{corollary}

\section{Examples and Discussion}

To compare known exponential stability results obtained in \cite{BB1,BB2} 
and the test obtained in Theorem \ref{theorem4.2},
consider equation (\ref{4.5}).

Let us illustrate our results with an example where Theorem~\ref{theorem4.2} implies exponential stability, while Propositions \ref{proposition5.1} and \ref{proposition5.2} fail. 
If anyone  of assumptions (\ref{5.1})-(\ref{5.3}) holds, conditions of Corollary~\ref{corollary4.2} also hold. 
However it is possible (see Example~\ref{example5.1}) that all conditions of Proposition~\ref{proposition5.1} fail, but conditions of Corollary \ref{corollary4.2} hold.
The paper \cite{BB1} also studies equation (\ref{1.1}) with several delay terms. Exponential stability tests in \cite{BB1} depend on the
greatest delay $\tau=\max\{\tau_k\}$ with the assumption that $b_k(t)\geq 0$. Theorem \ref{theorem4.1} depends on all delays $\tau_k$ without any assumption on the sign of $b_k(t)$.

Conditions in Proposition \ref{proposition5.2}, unlike in Corollary \ref{corollary4.2}, depend on the delay of a neutral term.
So for small $\sigma$ Proposition \ref{proposition5.2} is better than Corollary \ref{corollary4.2}, but for large $\sigma$, Corollary \ref{corollary4.2} is better (see again Example \ref{example5.1} below).

\begin{example}\label{example5.1}
Consider the equation
\begin{equation}\label{5.4}
\dot{x}(t)-0.15\dot{x}(t-\sigma)=-x \left( t-\frac{1}{e}-0.1 \sin{t} \right),~~\sigma>0.
\end{equation}
Here $t-h(t)=\frac{1}{e}+0.1\sin t\leq \tau:=\frac{1}{e}+0.1>\frac{1}{e}$. 
Conditions (\ref{5.1})-(\ref{5.3}) of Proposition~\ref{proposition5.1} fail, but conditions of Corollary \ref{corollary4.1} ($m=1$) hold. 
Hence by Corollary \ref{corollary4.1} equation (\ref{5.4}) is exponentially stable.
Assumptions of Proposition~\ref{proposition5.2} are satisfied for equation (\ref{5.4}) if $\sigma<2.165$. Thus, for $\sigma>2.2$, say, $\sigma=3$, the results of Corollary \ref{corollary4.2} lead to exponential stability of (\ref{5.4})  while Propositions~\ref{proposition5.1} and \ref{proposition5.2} cannot establish this result.  
\end{example}

We omit here comparison with other known stability results since \cite{BB1} contains this part.
Most of these results are for autonomous equations, or equations with a non-delay term, see for example \cite{Gop1, Gop2, KM}.
As we mentioned earlier, we are not aware of exponential estimates for solutions for neutral differential equations.
Exponential estimates for solutions of delay differential equations without a neutral term can be found in recent paper \cite{BB3}.
The present paper partially generalizes the results obtained in \cite{BB3} to the neutral case.

Let us illustrate exponential estimates for a solution of 
equation (\ref{5.4}) with either constant or variable $\sigma$.

\begin{example}\label{example5.2}
Consider the initial value problem
\begin{equation}\label{5.5}
\begin{array}{l}
\displaystyle
\dot{x}(t)-0.15\dot{x}(t-0.5)=-x \left( t-\frac{1}{e}-0.1\sin{t} \right),~ t\geq 0,\\
x(t)=\cos t, \dot{x}(t)=\sin {2t}+2, ~t<0, x(0)=1.
\end{array} 
\end{equation}
We apply Theorem \ref{theorem4.1}, where $m=1$, $a(t)=0.15$, $b(t)=1$, $\sigma=0.5$, $\displaystyle \frac{1}{e}-0.1\approx 
0.2679\leq t-h(t)\leq \tau=\frac{1}{e}+0.1\approx 0.4679$.
Thus inequality (\ref{4.1}) holds for $\lambda = 0.1$ with $M_1 \approx 0.96, M_0 \leq 25.61$. 
Hence for the fundamental function of the equation in (\ref{5.5})
we have an estimate $|X(t,s)|\leq 25.61 e^{-0.1(t-s)}$, $t \geq s\geq 0$.
Next, we have for the initial conditions
$\psi(t)=\sin {2t}+2$, $\varphi(t)=\cos t$, 
$\|\psi\|=3$, $\|\varphi\|=1$.
By  (\ref{4.3}), for the solution of problem (\ref{5.5}) we have the following estimate 
\begin{equation}\label{5.5a}
|x(t)|\leq 42.4 e^{-0.1t},
\end{equation}
see the comparison to the numerical solution in Fig.~\ref{figure1}, left.
Next, consider variable $\sigma$
\begin{equation}\label{5.6}
\begin{array}{l}
\displaystyle
\dot{x}(t)-0.15\dot{x}(t-2.7-0.3 \cos t)=-x \left( t-\frac{1}{e}-0.1\sin{t} \right)
,~ t\geq 0,\\
x(t)=\cos t, \dot{x}(t)=\sin {2t}+2, ~t<0, x(0)=1.
\end{array}
\end{equation}
All the parameters are as above, only instead of $\sigma=0.5$ we have $2.4 \leq t-\sigma(t) \leq 3$.
Inequality (\ref{4.1}) holds for $\lambda = 0.06$ with $M_0 \leq 25.5$. In this 
case, 
\begin{equation}\label{5.6a}
|x(t)|\leq 54.5 e^{-0.06 t},
\end{equation}
see the comparison to the numerical solution in Fig.~\ref{figure1}, right.

\begin{figure}[ht]
\centering
\includegraphics[scale=0.36]{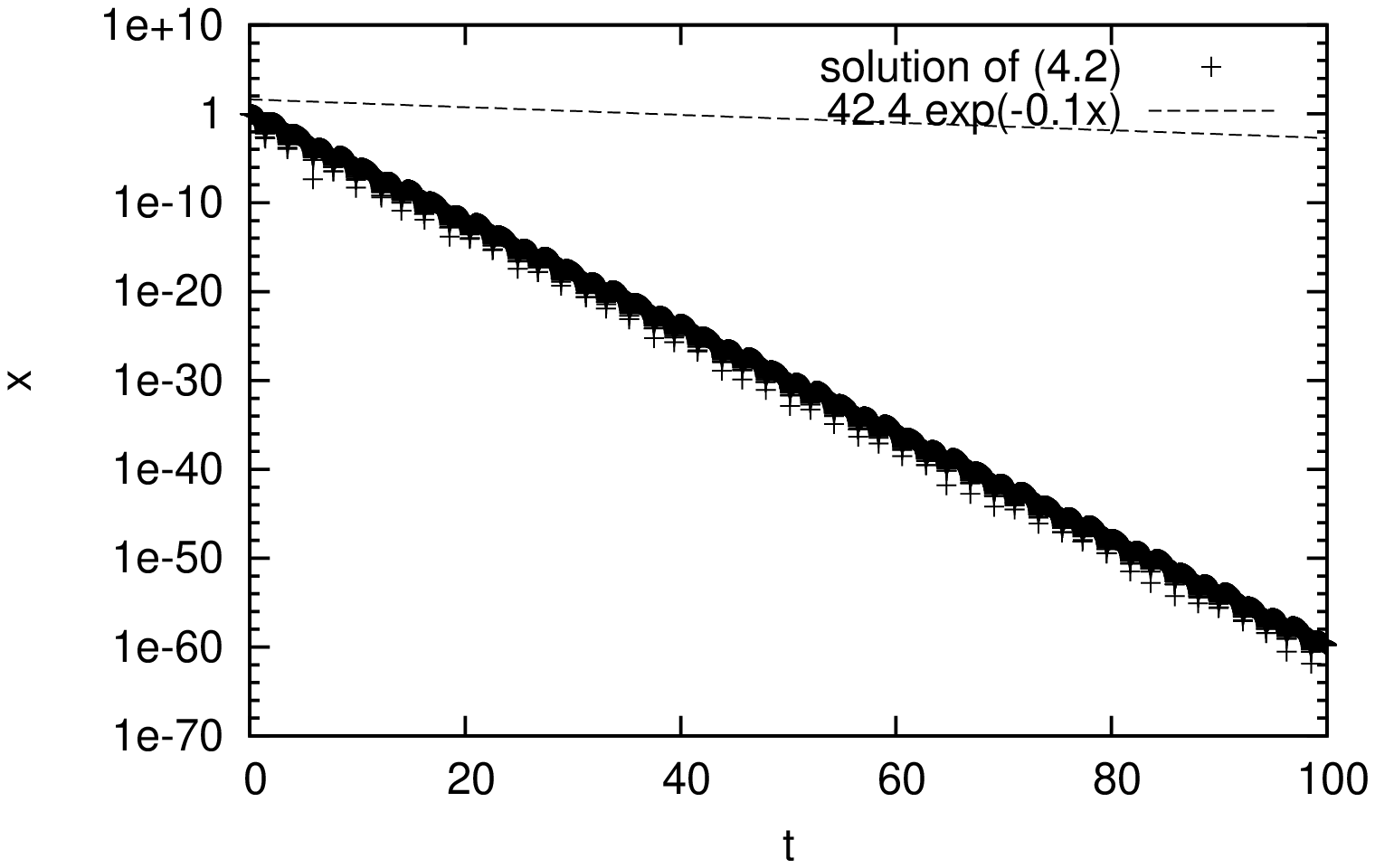}~~~~~~~~~~~~
\includegraphics[scale=0.36]{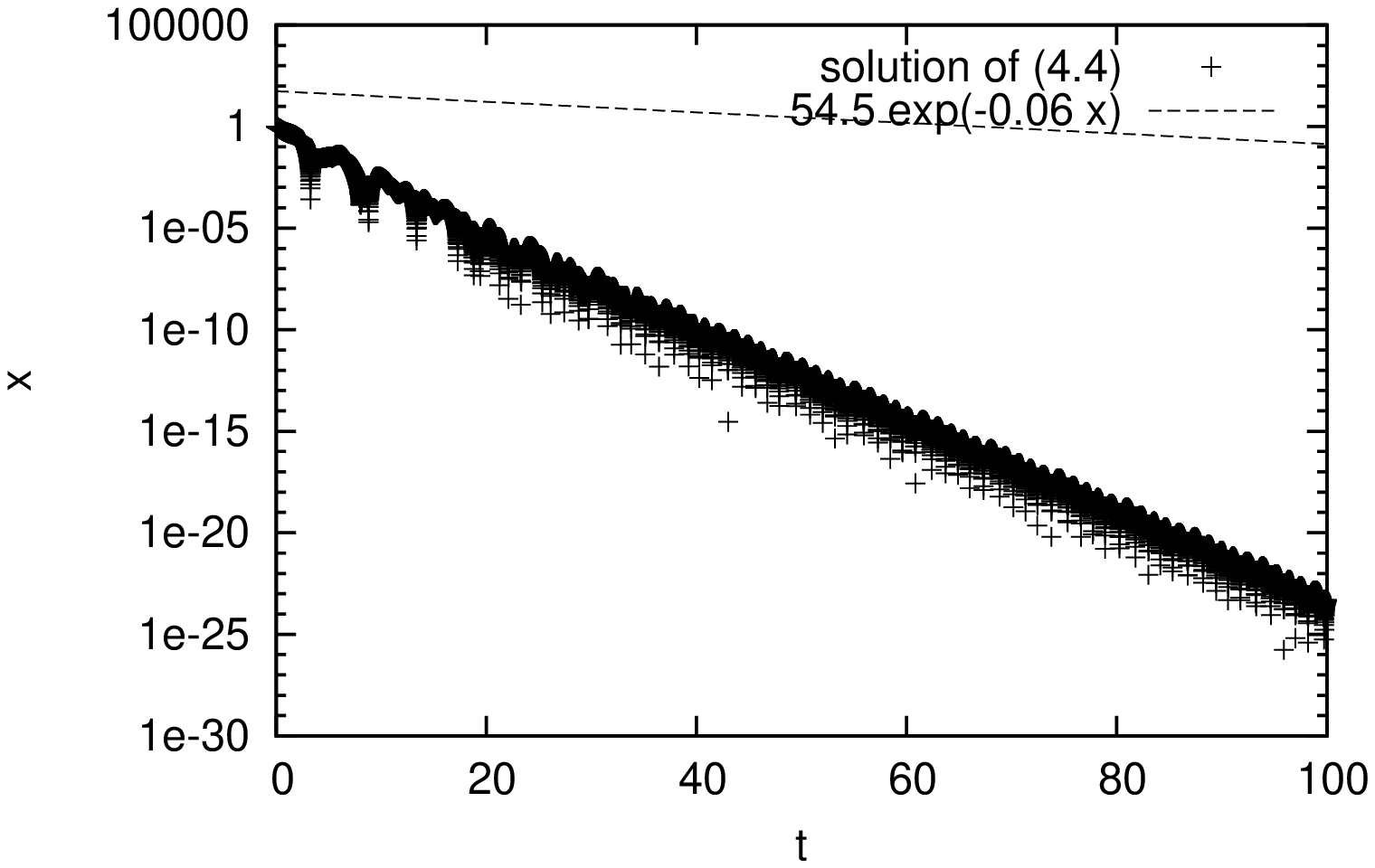}
\caption{The absolute value of the numerical solution of (\protect{\ref{5.5}}) compared to estimate 
(\protect{\ref{5.5a}}) (left) and the numerical solution of (\protect{\ref{5.6}}) compared to estimate 
(\protect{\ref{5.6a}}) (right), with the logarithmic scale in $x$.}
\label{figure1}
\end{figure}
\end{example}

Finally, let us state possible directions for further research extending the results of the present paper.

\begin{enumerate}
\item
Obtain explicit estimates of solutions for nonlinear neutral differential equations.
\item
Extend the estimates of solutions to a vector DDE,  or to higher order neutral equations, for recent results on third order neutral equations see \cite{Dom2019}. Consider other types of neutral equations, such as equations with a distributed delay, and stochastic differential equations.
\item
In this paper, we presented pointwise estimates. It would be interesting to obtain estimates in an integral form.
\item
In Corollary \ref{corollary4.1} the right-hand side is equal to $1-2\|a\|$ implying $\|a\|<\frac{1}{2}$. Can we extend the results to $\|a\| \in (0.5,1)$?
\item
Derive exponential estimates dependent on the neutral delay $\sigma=\esssup_{t\geq t_0} (t-g(t))$ for problem \eqref{2.1}-\eqref{2.2}.
\end{enumerate}

\section*{Acknowledgment}

The second author acknowledges the support of NSERC, the grant RGPIN-2020-03934. 

\end{document}